\documentclass[11pt,a4paper]{article}
\usepackage[utf8]{inputenc}
\usepackage{microtype}
\usepackage{amsmath,amsthm,amssymb}
\usepackage{mathrsfs} 
\usepackage[T1]{fontenc}
\usepackage{ae,aecompl}
\usepackage{times}
\usepackage[british]{babel}
\usepackage{aliascnt}
\usepackage[colorlinks,pagebackref,hypertexnames=false]{hyperref} 
\usepackage{url}
\usepackage[alphabetic,backrefs]{amsrefs}
\usepackage{nth}
\usepackage{float}
\usepackage[all]{xy}
\usepackage{tikz-cd} 
\usepackage{bookmark}
\usepackage[margin=1.9cm]{geometry}
\usepackage{enumerate}

\newcommand{\rd}{{\rm d}}

\newcommand{\ri}{{\rm i}}
\newcommand{\rj}{{\rm j}}

\newcommand{\rp}{{\rm p}}
\newcommand{\rmq}{{\rm q}}
\newcommand{\rr}{{\rm r}}

\newcommand{\rG}{{\rm G}}

\newcommand{\oep}{{\overline {\Upsilon}}}

\newcommand{\cD}{\mathcal{D}}

\newcommand{\cL}{\mathcal{L}}

\newcommand{\Q}{\mathbb{Q}}
\newcommand{\R}{\mathbb{R}}
\newcommand{\C}{\mathbb{C}}

\newcommand{\SU}{{\rm SU}}
\newcommand{\GL}{\mathrm{GL}}

\newcommand{\Curl}{\mathrm{Curl}}

\renewcommand{\P}{\mathbb{P}}

\renewcommand{\epsilon}{\varepsilon}

\newcommand{\Hol}{\mathrm{Hol}}

\newcommand{\Ric}{{\rm Ric}}

\renewcommand{\Im}{\mathop{\mathrm{Im}}}

\renewcommand{\Re}{\mathop{\mathrm{Re}}}

\newcommand{\tr}{\mathop{\mathrm{tr}}\nolimits}

\newcommand{\vol}{\mathrm{vol}}

\newcommand{\gt}{\texorpdfstring{\mathrm{G}_2}{\space}}

\renewcommand{\div}{\mathrm{div}}

\newcommand{\qandq}{\quad\text{and}\quad}

\def\<{\mathopen{}\left<}
\def\>{\right>\mathclose{}}
\def\({\mathopen{}\left(}
\def\){\right)\mathclose{}}

\usepackage{multicol, color}

\definecolor{gold}{rgb}{0.85,.66,0}
\definecolor{cherry}{rgb}{0.9,.1,.2}
\definecolor{burgundy}{rgb}{0.8,.2,.2}
\definecolor{orangered}{rgb}{0.85,.3,0}
\definecolor{orange}{rgb}{0.85,.4,0}
\definecolor{olive}{rgb}{.45,.4,0}
\definecolor{lime}{rgb}{.6,.9,0}
\definecolor{green}{rgb}{.2,.7,0}
\definecolor{grey}{rgb}{.4,.4,.2}
\definecolor{brown}{rgb}{.4,.3,.1}

\newtheorem{thmx}{Theorem}

\newtheorem{theorem}{Theorem}
\newtheorem{prop}{Proposition}

\newtheorem{lemma}[prop]{Lemma}

\numberwithin{substep}{step}

\numberwithin{subcase}{case}

\theoremstyle{remark}
\newtheorem{remark}{Remark}

\theoremstyle{definition}
\newtheorem{definition}{Definition}
\newtheorem{example}{Example}

\usepackage{todonotes}
\usepackage{comment}

\numberwithin{equation}{section}
\numberwithin{theorem}{section}
\numberwithin{prop}{section}
\numberwithin{cor}{section}
\numberwithin{definition}{section}
\numberwithin{remark}{section}
\numberwithin{example}{section}

\allowdisplaybreaks

\usepackage{titling}
\usepackage{changepage}

\begin{document}
	\title{Flows of \texorpdfstring{G\textsubscript{2}}{G2}-structures on contact Calabi--Yau $7$-manifolds}
	\author{
	Jason D. Lotay \\
    \emph{University of Oxford}
    \and 
    Henrique N. S\'a Earp 
    \qquad\qquad
    Julieth Saavedra
    \\
  \emph{University of Campinas (Unicamp)}}
  \date{}

\maketitle
	
\vspace{-0.5cm}	
\begin{abstract}
    We study the Laplacian flow and coflow on contact Calabi-Yau $7$-manifolds.  We show that the natural initial condition leads to an ancient solution of the Laplacian flow with a finite time Type I singularity which is not a soliton, whereas it produces an immortal (though neither eternal nor self-similar) solution of the Laplacian coflow which has an infinite time singularity of Type IIb, unless the transverse Calabi--Yau geometry is flat.  The flows in each case collapse (under normalised volume) to a lower-dimensional limit, which is either $\mathbb{R}$, for the Laplacian flow, or standard $\mathbb{C}^3$, for the Laplacian coflow.  We also study the Hitchin flow in this setting, which we show coincides with the Laplacian coflow, up to reparametrisation of time, and defines an (incomplete) Calabi--Yau structure on the spacetime track of the flow.
\end{abstract}
	
\begin{adjustwidth}{0.95cm}{0.95cm}
    \tableofcontents
\end{adjustwidth}
	
\section{Introduction}
	
Geometric flows have proven to be a powerful analytical tool in a variety of geometric and topological problems. In the case of $\gt$-geometry, they provide a method to search for  metrics with $\gt$ holonomy, which are then Ricci-flat, on an oriented and spin $7$-manifold $M$. This corresponds to varying a $\gt$-structure, given by a non-degenerate $3$-form $\varphi$ on $M$, so that it becomes torsion-free. Flows in $\gt$-geometry were first introduced by Bryant \cite{Bryant2006} and have been studied by several authors, see e.g.~\cites{Bryant2011, Karigiannis2012,   Lauret2017, Fino2018, Grigorian2019,Dwivedi2019, Karigiannis2020,  Loubeau2019, Lotay2015, Loubeau2021}. 
	
A $\gt$-structure $\varphi$ determines a  metric $g_{\varphi}$ and orientation with Riemannian volume form $\vol_{\varphi}$. 
The \emph{full torsion tensor} $T$ of $\varphi$ is a 2-tensor which is equivalent to $\nabla^{g_\varphi}\varphi$, where $\nabla^{g_{\varphi}}$ is the Levi-Civita connection of $g_{\varphi}$.  Pairs $(M^7,\varphi)$ which satisfy $T\equiv0$ (i.e.~so that $\varphi$ is torsion-free) are called $\gt$-\emph{manifolds} and are of particular interest since the holonomy group of $g_{\varphi}$ in this case is contained in $\gt$.  However,  complete examples of $\gt$-manifolds $(M,\varphi)$ are very difficult to construct, especially when  $M$ is required to be compact.  
Fernandez and Gray \cite{Fernandez1982} showed that the torsion-free condition is equivalent  to $\varphi$ being both \emph{closed} and \emph{coclosed}, i.e,  $\rd\varphi=0$ and $\rd\!\ast_\varphi\!\varphi=0$, where $*_{\varphi}$ is the Hodge star defined by $g_{\varphi}$ and $\vol_{\varphi}$.  This alternative viewpoint on the torsion-free condition as a system of nonlinear PDE is fundamental to $\gt$ geometry and to geometric flows.

\subsection{Laplacian coflow}\label{ss:L.coflow}	
Our first goal is to study the  \emph{Laplacian coflow} of $\gt$-structures, introduced by Karigiannis, McKay and Tsui \cite{Karigiannis2012}\footnote{In \cite{Karigiannis2012}, the flow \eqref{eq: Laplacian.coflow} is written with an additional minus sign on the right-hand side, but this seems incorrect given the work on the Laplacian flow \cites{Bryant2011,Lotay2017} and the modified Laplacian coflow \cite{Grigorian2013}.} 
which is given by 
\begin{equation}
\label{eq: Laplacian.coflow}
	\frac{\partial\psi_t}{\partial t}
	=\Delta_{\psi_t}\psi_t
	:=(\rd\rd^{\ast_t}+\rd^{\ast_t}\rd)\psi_t,
\end{equation}
where $\psi_t:=\ast_t\varphi_t$ is the Hodge dual of the $\gt$-structure $\varphi_t$ (here, we write $*_t:=*_{\varphi_t}$) and $\Delta_{\psi_t}$ 
is the Hodge Laplacian of $g_{t}:=g_{\varphi_t}$ on $4$-forms.  
If $M$ is compact, critical points of this flow \eqref{eq: Laplacian.coflow} are then Hodge duals of non-degenerate 3-forms which are closed and coclosed, i.e.~torsion-free.

We  restrict the Laplacian coflow  \eqref{eq: Laplacian.coflow} to the case where $\psi_t$ is closed, i.e.~the $\gt$-structure $\varphi_t$ is coclosed. In this context, solutions of  \eqref{eq: Laplacian.coflow} (if they exist) preserve the cohomology class  $[\psi_t]=[\psi_0]\in H^4(M)$, for all $t$, and the flow seeks critical points in this class.  When $M$ is compact, \eqref{eq: Laplacian.coflow} can be interpreted as the gradient flow  of \emph{Hitchin's volume functional \cite{Hitchin2001a}} and so the volume of $M$ increases monotonically along the flow (see \cite{Grigorian2013}).
	
One immediate problem with the Laplacian coflow is that the $4$-form $\psi=*_{\varphi}\varphi$ is generated by both the 3-forms $\varphi$ and $-\varphi$: in particular, $\psi$ does not determine the orientation on $M$. 
\color{black} However, it is natural to fix an orientation throughout the flow, which is determined for example by a choice of $\gt$-structure dual to the initial 4-form. \color{black} 
Another key problem, which is much more serious, is that it is not known whether solutions to \eqref{eq: Laplacian.coflow} actually exist in general, even for an arbitrarily short time.  There is a modification of the coflow \cite{Grigorian2013} which does have guaranteed short-time existence, but the critical points are no longer necessarily closed and coclosed 4-forms, and so it does not currently appear to be useful as a tool for studying key problems in $\gt$ geometry.

Finally, even if one has short-time existence and uniqueness of the flow \eqref{eq: Laplacian.coflow}, one finds that the torsion tensor is not guaranteed to satisfy a heat-type equation.  Thus the Laplacian coflow is not a \emph{reasonable} (sometimes also called \emph{Ricci-like}) flow of $\rG_2$-structures in the sense of \cite{Chen2018},  so the analytic theory developed there does not apply. 
This issue is linked to the lack of parabolicity of the Laplacian coflow, even in the direction of closed 4-forms (see \cite{Grigorian2013}).

In Theorem \ref{thm:Laplacian.coflow.cCY}, we will find immortal families of coclosed $\gt$-structures solving the Laplacian coflow \eqref{eq: Laplacian.coflow}.

\subsection{Laplacian flow}

A geometric flow of $\gt$-structures which has received much more attention is the Laplacian flow introduced by Bryant \cite{Bryant2006}:
\begin{equation}
\label{eq: Laplacian.flow}
    \frac{\partial\varphi_t}{\partial t}=\Delta_{\varphi_t}\varphi_t:=(\rd\rd^{*_t}+\rd^{*_t}\rd)\varphi_t,
\end{equation}
where we again write $*_t$ for the Hodge star induced by the $\gt$-structure $\varphi_t$, and $\Delta_{\varphi_t}$ is the Hodge Laplacian of $g_t:=g_{\varphi_t}$ on 3-forms.  Again, if $M$ is compact, the Laplacian flow \eqref{eq: Laplacian.flow} has torsion-free $\gt$-structures as its critical points.

Typically, one restricts the Laplacian flow \eqref{eq: Laplacian.flow} to \emph{closed} $\gt$-structures, because in that setting one obtains a good analytic theory, the cohomology class of the flowing $\gt$-structure $\varphi_t$ is preserved (so we are seeking torsion-free $\gt$-structures in a given class), and the flow (in the compact setting) can be interpreted as the gradient flow of the Hitchin volume functional on $[\varphi_t]=[\varphi_0]$.  See \cites{Bryant2011, Karigiannis2020, Lotay2017, Lotay2015} for more details on the Laplacian flow.
By contrast, it is not known whether starting \eqref{eq: Laplacian.flow} at a \emph{coclosed} $\gt$-structure may preserve coclosedness, however much one might naively expect this from a flow which is initially in the direction of coexact forms.
\color{black} Indeed, to our knowledge, there are no particular examples in the literature of a Laplacian flow preserving the coclosed property. On the other hand, neither are there explicit examples for which the coclosed condition is not preserved, as far as the authors are aware, so the field seems still very much open to exploration. \color{black}   
 We also note that we do not currently have any general analytic theory for the Laplacian flow, except when restricted to closed $\gt$-structures.  

In Theorem \ref{thm:Laplacian.flow.cCY}, we will obtain ancient families of coclosed $\gt$-structures solving the Laplacian flow \eqref{eq: Laplacian.flow}.

\subsection{Singularities}

For various flows of $\gt$-structures (cf.~\cites{Lotay2017,Chen2018}) it has been shown that blow-up of the following quantity characterises the formation of finite-time singularities: 
\begin{equation*}
    \Lambda(x,t)=(|Rm(x,t)|^2_{g_t} +|T(x,t)|^4_{g_t} +|\nabla^{g_t} T(x,t)|^2_{g_t})^{\frac{1}{2}}
\end{equation*}
for $x\in M$ and time $t$.  We then let
\begin{equation}\label{eq:Lambda.t}
    \Lambda(t)=\sup_{x\in M}\Lambda(x,t)
\end{equation}
for convenience and introduce the following terminology, motivated by the singularity classification in Ricci flow (cf.~\cite{Hamilton2006}) and the work in \cites{Lotay2017,Chen2018}. 

\begin{definition}\label{dfn:sing.types}
    Suppose that $(M^7,\varphi_t,\psi_t,g_t)$ is a solution to a flow of $\gt$-structures on a closed manifold on a maximal time interval $[0,T)$ and let $\Lambda(t)$ be as in \eqref{eq:Lambda.t}.
    
    If we have a finite-time singularity, i.e.~$T<\infty$, we say that the solution forms  
\begin{itemize}
    \item a \emph{Type I singularity} (rapidly forming) if $\sup_{t\in[0,T)}(T-t)\Lambda(t) <\infty$; and otherwise 
    \item a \emph{Type IIa singularity} (slowly forming) if $\sup_{t\in [0,T)}(T-t)\Lambda(t)=\infty$.
\end{itemize}

If we have an \emph{infinite-time} singularity, where $T=\infty$, then it is 
\begin{itemize}
    \item a \emph{Type IIb singularity} (slowly forming) if $\sup_{t\in  [0,\infty)}t\Lambda(t)=\infty$; and otherwise  
    
    \item a \emph{Type III singularity} (rapidly forming) if $\sup_{t\in  [0,\infty)}t\Lambda(t)<\infty$.
\end{itemize}
\end{definition}

In Theorem \ref{thm:sings}, we will apply this singularity classification to our solutions to the flows \eqref{eq: Laplacian.coflow} and  \eqref{eq: Laplacian.flow}.

\subsection{Hitchin flow}

The \emph{Hitchin flow}\footnote{It should be made clear that the Hitchin flow for $\gt$-structures is not a geometric flow in the usual sense, but it is nonetheless a valid evolution equation for $\gt$-structures.} \cite{Hitchin2001a}  for a family $\varphi_t$ of  $\gt$-structures on $M^7$ with dual 4-form $\psi_t$  is given by solving
\begin{equation}
\label{eq:Hitchin.flow}
    \frac{\partial\psi_t}{\partial t}=\rd\varphi_t\qandq \rd\psi_t=0
\end{equation}
for $t\in I\subseteq\mathbb{R}$.  We note that the Hitchin flow \eqref{eq:Hitchin.flow} is an evolution equation for \emph{coclosed} $\gt$-structures.  The significance of the flow is that a solution defines a $4$-form $\Phi$ on the product $I\times M$ by
\begin{equation}
    \Phi=\rd t\wedge\varphi_t+\psi_t,
\end{equation}
which is closed and so defines a torsion-free $\mathrm{Spin}(7)$-structure on $I\times M$ (and thus a Ricci-flat metric with holonomy contained in $\mathrm{Spin}(7)$).  
Relatively little is known about the Hitchin flow and since the 7-manifolds we are studying naturally admit coclosed $\gt$-structures, it is worth examining the Hitchin flow in this context.

In Theorem \ref{thm:Hitchin.flow}, we will obtain immortal solutions to the Hitchin flow \eqref{eq:Hitchin.flow}, which bear a close relation to our Laplacian coflow \eqref{eq: Laplacian.coflow} solutions, and define an incomplete \emph{Calabi-Yau} structure on the 8-dimensional spacetime track of the flow.

\subsection{Outline and main results}
\label{sec: main results}

We  will consider the flows described above on a \emph{contact Calabi--Yau (cCY)} 7-manifold $(M^{7},g,\eta,\Upsilon)$, where $(M^7,g)$ is a Sasakian 7-manifold with Riemannian metric $g$, contact form $\eta$ and transverse K\"ahler form $\omega=\rd\eta\in\Omega^{1,1}(M)$, and $\Upsilon\in\Omega^{3,0}(M)$ is a  transverse holomorphic volume form; here $(p,q)$ denotes basic bidegree with respect to the horizontal distribution $\cD=\ker\eta$, cf.~\S\ref{sec: cCY manifolds}. An important class of compact  examples of cCY 7-manifolds are \emph{Calabi--Yau links}, which are total spaces of $S^1$-(orbi)bundles over  Calabi--Yau $3$-orbifolds famously listed by Candelas-Lynker-Schimmrigk \cite{Candelas1990}; these   will be discussed in Example \ref{ex: CY links}.

On a cCY 7-manifold there exists a natural 1-parameter family of coclosed $\gt$-structures defined, for each $\epsilon>0$, by  
\begin{equation}
\label{eq:std.phi}
	\varphi=\epsilon\eta\wedge\omega+\Re\Upsilon,
\end{equation} 
with induced metric $g_{\varphi}$ (which equals $g$ when $\epsilon=1$) and corresponding dual $4$-form
\begin{equation}
\label{eq:std.psi}
	\psi=\ast_{\varphi}\varphi=\frac{1}{2}\omega^2 -\epsilon\eta\wedge\Im\Upsilon.
\end{equation}
In the present paper, we are interested in flows of $\gt$-structures on cCY 7-manifolds starting at   the natural coclosed $\gt$-structure in \eqref{eq:std.phi} or its dual 4-form in \eqref{eq:std.psi}. All of our results are stated in terms of the following standard data \eqref{eq: standard setup}:
\begin{equation}
\label{eq: standard setup}
\tag{S}
\text{\parbox{.90\textwidth}
    {$\bullet$ $(M^7,g_0,\eta_0,\Upsilon_0)$ a contact Calabi-Yau 7-manifold;\\
    $\bullet$ $\cD_0=\ker\eta_0$ its horizontal distribution, and $\omega_0=\rd\eta_0$ its transverse Kähler form;\\
    $\bullet$ for $\epsilon>0$, the $\gt$-structure defined by \eqref{eq:std.phi} and \eqref{eq:std.psi}, i.e. 
    $$
    \varphi_0=\epsilon\eta_0\wedge\omega_0+\Re\Upsilon_0
    \qandq
    \psi_0=\frac{1}{2}\omega_0^2 -\epsilon\eta_0\wedge\Im\Upsilon_0.$$
}}
\end{equation}

In \S\ref{Section.Laplacian.coflow}, we consider the Laplacian coflow  on a cCY 7-manifold and prove the following. 

\begin{thmx}[Laplacian coflow on contact Calabi--Yau 7-manifolds]
\label{thm:Laplacian.coflow.cCY} 
In the setup \eqref{eq: standard setup}, the Laplacian coflow \eqref{eq: Laplacian.coflow}  on $M^7$, with initial data determined by $\varphi_0$, is solved by the following family of coclosed $\gt$-structures $\varphi_t$, with associated metric $g_t$, volume form $\vol_t$ and dual 4-form $\psi_t$:
    \begin{align*}
	\varphi_t&= \epsilon\rp(t)^{-1}\eta_0\wedge\omega_0+\rp(t)^{3}\Re\Upsilon_0;
	\\
	\psi_t&= \frac{1}{2}\rp(t)^{4}\omega_0^2-\epsilon\eta_0\wedge\Im\Upsilon_0;
	\\
	g_t&= \epsilon^2\rp(t)^{-6}\eta_0^2+\rp(t)^{2}g_{\cD_0};\\
	\vol_t&=\epsilon\rp(t)^3\eta_0\wedge\vol_{\cD_0},
\end{align*}
    where $\rp(t)=(1+10\epsilon^2t)^{1/10}$ and $t\in (-\frac{1}{10\epsilon^2},\infty)$.  
    Hence, the solution of the Laplacian coflow is immortal, with a finite time singularity (backwards in time) at 
    $t=-\frac{1}{10\epsilon^2}$.
\end{thmx}

\begin{remark}
    If $M^7$ is compact then the volume of $M$ determined by the $\gt$-structure $\varphi$ on $M$ is:
$$
\mathcal{H}(\varphi):=\textrm{Vol}(M,\varphi)=\frac{1}{7}\int_M\varphi\wedge \psi.
$$
    Along the Laplacian coflow solution given in Theorem \ref{thm:Laplacian.coflow.cCY} for a compact cCY 7-manifold we have that
$$
\mathcal{H}(\varphi_t)=(10t+1)^{3/10}\mathcal{H}(\varphi_0).
$$
    Hence, the Hitchin functional on the cohomology class $[\psi_0]$, which is just $\mathcal{H}(\varphi_t)$, tends to infinity as $t\to\infty$ and tends to $0$ as $t\to-1/10$.  In particular, recalling that the Laplacian coflow is the gradient flow of the Hitchin functional on $[\psi_0]$, we observe that the Hitchin functional is unbounded above and does not have a positive lower bound on $[\psi_0]$.
\end{remark}

In $\S$\ref{Section.Laplacian.flow}, we switch attention to the Laplacian flow and prove the following result.

\begin{thmx}[Laplacian flow on contact Calabi--Yau 7-manifolds]
\label{thm:Laplacian.flow.cCY}
    In the setup \eqref{eq: standard setup}, the Laplacian flow \eqref{eq: Laplacian.flow} on $M^7$, with initial data determined by $\varphi_0$, is solved by the following family of \emph{coclosed} $\gt$-structures $\varphi_t$ on $M^7$, with associated metric $g_t$, volume form $\vol_t$ and dual 4-form $\psi_t$ given by
    \begin{align*}
    \varphi_t&=\epsilon\rmq(t)\eta_0\wedge\omega_0+\Re\Upsilon_0;\\
    \psi_t&=\frac{1}{2}\omega_0^2-\epsilon\rmq(t)\eta_0\wedge\Im\Upsilon_0;\\
    g_t&=\epsilon^2\rmq(t)^2\eta_0^2+g_{\cD_0};\\
    \vol_t&=\epsilon\rmq(t)\eta_0\wedge\vol_{\cD_0},
\end{align*}
    where $\rmq(t)=(1-8\epsilon^2t)^{-1/2}$ and $t\in(-\infty,\frac{1}{8\epsilon^2})$.  
    Hence, the solution of the Laplacian flow is ancient, with a finite time singularity (forwards in time) at $t=\frac{1}{8\epsilon^2}$.
\end{thmx}

\begin{remark}
Here, we find that the Laplacian flow solution in Theorem \ref{thm:Laplacian.flow.cCY} has the property that $[\psi_t]$ is not constant unless $[\psi_0]=0$.  Moreover, when $[\psi_0]\neq 0$, one can detect the finite time singularity of the flow using the cohomology class $[\psi_t]$, which is somewhat reminiscent of the K\"ahler--Ricci flow.
\end{remark}

\begin{remark}
We can interpret taking $\epsilon\to 0$ in Theorems \ref{thm:Laplacian.coflow.cCY} and \ref{thm:Laplacian.flow.cCY} as obtaining degenerate eternal solutions to the Laplacian flow or coflow, which is simply given by the Calabi--Yau structure on the horizontal distribution $\cD_0$.  In the quasi-regular setting for the contact Calabi--Yau structure, where $M$ fibres by circles over a Calabi--Yau 3-orbifold $Z$, sending $\epsilon$ to $0$ is equivalent to shrinking the circle fibres and collapsing to $Z$.

If we instead consider $\epsilon\to\infty$ in Theorems \ref{thm:Laplacian.coflow.cCY} and \ref{thm:Laplacian.flow.cCY}, we appear to obtain a degenerate solution to the flow, which is only defined for non-negative times for the coflow and for non-positive times for the flow.  Again, in the quasi-regular setting, this would be equivalent to expanding the circle fibres to infinite size (usually interpreted as becoming lines) over the Calabi--Yau 3-orbifold base. 
\end{remark}

When $M$ is compact, we can describe the singularities of our flow solutions in terms of Definition \ref{dfn:sing.types}.

\begin{thmx}[Singularities of the Laplacian flow and coflow]
\label{thm:sings}
Suppose we are in the setup \eqref{eq: standard setup} and $M$ is compact. 
\begin{itemize}
    \item[(a)] The Laplacian coflow solution in Theorem \ref{thm:Laplacian.coflow.cCY} has an infinite-time Type IIb singularity, unless the transverse metric on $\cD_0$ is flat, in which case it has an infinite-time Type III singularity.
    \item[(b)] The Laplacian flow solution in Theorem \ref{thm:Laplacian.flow.cCY} has a finite-time Type I singularity.
\end{itemize}    
\end{thmx}

\begin{remark}
This theorem gives the first examples of compact solutions to the Laplacian coflow which have an infinite-time Type IIb singularity.  It also provides the first compact solutions to the Laplacian flow which have a finite-time singularity, and the first Type I singularity which is \emph{not} a soliton. 
 \color{black} Finite-time \color{black} singularities are not yet known to occur in the Laplacian flow for \emph{closed} $\gt$-structures on compact manifolds, which moreover are not expected to be Type I because there are no compact shrinking solitons in this setting.  
\end{remark}

In Section \ref{sec: Hitchin.flow} we study the Hitchin flow for the natural coclosed $\gt$-structures on cCY 7-manifolds, and obtain the following result.

\begin{thmx}[Hitchin flow on contact Calabi--Yau 7-manifolds]
\label{thm:Hitchin.flow}
    In the setup \eqref{eq: standard setup}, the Hitchin flow \eqref{eq:Hitchin.flow} on $M^7$, with initial data determined by $\varphi_0$, is solved by the following family of coclosed $\gt$-structures $\varphi_t$, with associated metric $g_t$, volume $\vol_t$ and dual 4-form $\psi_t$ given by
    \begin{align*}
    \varphi_t&=\epsilon\rr(t)^{-1}\eta_0\wedge\omega_0+\rr(t)^3\Re\Upsilon_0;\\
    \psi_t&=\frac{1}{2}\rr(t)^4\omega_0^2-\epsilon\eta_0\wedge\Im\Upsilon_0;\\
    g_t&=\epsilon^2\rr(t)^{-6}\eta_0^2+\rr(t)^2g_{\cD_0};\\
    \vol_t&=\epsilon\rr(t)\eta_0\wedge\vol_{\cD_0},
\end{align*}
where $\rr(t)=(1+\frac{5}{2}\epsilon t)^{1/5}$ and $t\in(-\frac{2}{5\epsilon},\infty)$. This Hitchin flow solution coincides with the Laplacian coflow in Theorem \ref{thm:Laplacian.coflow.cCY}, up to reparametrisation of time,
and it defines an incomplete Calabi--Yau structure on $(-\frac{2}{5\epsilon},\infty)\times M$, with K\"ahler form $\widehat{\omega}$, metric $\widehat{g}$ with $\Hol(\widehat{g})\subseteq \SU(4)$, and holomorphic volume form $\widehat{\Upsilon}$, as follows:
\begin{align*}
    \widehat{\omega}&=\epsilon\rr(t)^{-3}\rd t\wedge\eta_0+\rr(t)^2\omega_0;\\
    \widehat{g}&=\rd t^2+\epsilon^2\rr(t)^{-6}\eta_0^2+\rr(t)^2g_0;\\
    \widehat{\Upsilon}&=(\rr(t)^3\rd t+i\epsilon\eta_0)\wedge \Upsilon_0.
\end{align*}
\end{thmx}

\begin{remark} The relation in Theorem \ref{thm:Hitchin.flow} between contact Calabi--Yau 7-manifolds, the Hitchin flow and Calabi--Yau structures on 8-manifolds was previously known, for example, by work in \cite{ContiFino} and \cite{Freibert}.  In the quasi-regular setting, the Calabi--Yau structures in  Theorem \ref{thm:Hitchin.flow} arise via the well-known Calabi ansatz on open subsets of complex line bundles over Calabi--Yau (or, more generally, K\"ahler--Einstein) orbifolds.  It is interesting however to observe the close connection between the Laplacian coflow and the Hitchin flow in this setting, which is a novel aspect of our study.  \color{black} We have no reason to expect such a relation between these flows to occur in general.\color{black}
\end{remark}

\begin{remark}
A natural question, if perhaps a little bold, is whether flows on cCY $7$-manifolds might actually converge to a torsion-free $\gt$-structure. While cautioning against excessive optimism, we notice that there are many compact examples with no apparent topological obstructions, on which moreover such a structure would induce a metric with holonomy group $\gt$, cf.~Remark \ref{rem: optimist topology} in \S\ref{sec: cCY manifolds}.  We emphasise that the study in this article does not pertain to this question, because the  $\gt$-structures under consideration admit a non-trivial Killing field and so could not converge under the flow to a torsion-free $\gt$-structure.
\end{remark}

\bigskip 

\noindent\textbf{Acknowledgements:} The authors would like to thank Simon Salamon, Mark Haskins and Andr\'es Moreno  for some valuable discussions.
JL and HSE were supported by a UK Royal Society Newton Mobility Award [NMG$\backslash$R1$\backslash$191068].
JL is also partially supported by the Simons Collaboration on Special Holonomy in Geometry, Analysis, and Physics ($\#724071$ Jason Lotay). 
HSE has also been supported by the São Paulo Research Foundation (Fapesp)  \mbox{[2018/21391-1]} and the Brazilian National Council for Scientific and Technological Development (CNPq)  \mbox{[311128/2020-3]}.
JPS was supported by the Coordination for the Improvement of Higher Education Personnel-Brazil (CAPES) [88882.329037/2019-1].

\bigskip

\noindent\textbf{Data availability statement:} Data sharing not applicable to this article as no datasets were generated or analysed during the current study.

\section{Preliminaries on coclosed \texorpdfstring{G\textsubscript{2}}{G2}-structures}
\label{Section.coclosed}

In this section, we collect definitions and important properties regarding (coclosed) $\gt$-structures which will be useful in the study of flows on contact Calabi--Yau 7-manifolds.  

\subsection{\texorpdfstring{G\textsubscript{2}}{G2}-structures and their torsion forms}

A $\gt$-structure on a (orientable and spin) 7-manifold $M$ is a reduction of the structure group  of $TM$ to $\gt$, where $\gt$ is viewed as the subgroup of $\GL(7,\R)$ containing all linear maps preserving the $3$-form 
$$ 
\varphi_0
=e^{123}+e^{145}+e^{167}+e^{246}-e^{257}-e^{347}-e^{356}
\in\Lambda^3(\R^7)^*,
$$ 
where $e^i=dx^i$,  $e^{ij}=e^i\wedge e^j$ etc. It is equivalently determined by a $3$-form  $\varphi$ on $M$ such that $(TM,\varphi)$ is pointwise isomorphic to $(\R^7,\varphi_0)$. 
A $\gt$-structure $\varphi$ determines a Riemannian metric $g_{\varphi}$ and an orientation given by the Riemannian volume form $\vol_{\varphi}$ so that
$$
	6g_\varphi(X,Y)  \vol_{\varphi}
	=(X\lrcorner \varphi)\wedge (Y\lrcorner\varphi)\wedge\varphi
$$
for any tangent vectors $X,Y$ on $M$.  For simplicity, we  write $g=g_{\varphi}$. The metric  and the orientation determine a Hodge star operator $\ast_{\varphi}$, so we  have the dual $4$-form $\psi=\ast_{\varphi}\varphi$ of the $\gt$-structure $\varphi$.  

A $\gt$-structure gives rise to a decomposition of differential forms on $M$ corresponding to irreducible $\gt$ representations. In particular, the spaces $\Omega^2$ and $\Omega^3$ of $2$-forms and $3$-forms decompose as 
\begin{align}
\label{eq:form.decomp}
	\Omega^2 &= \Omega^2_7\oplus\Omega^2_{14}\qandq
	\Omega^3  = \Omega^3_1\oplus\Omega_{7}^{3}\oplus\Omega^3_{27},
\end{align}
where $\Omega^k_l$ has (pointwise) dimension $l$ and this decomposition is orthogonal with respect to the metric $g$.  Via the Hodge star, this defines corresponding decompositions of the 4-forms  $\Omega^4$ and 5-forms $\Omega^5$.  If we let $S^2$ denote the symmetric 2-tensors on $M$ then, as in \cite{Bryant2006}, we define a linear operator $ \rj_{\varphi}:\Omega^3\rightarrow S^2$ by
\begin{align}
\label{Eq:j.operator}
	\rj_{\varphi}(\gamma)(X,Y)
	&=\ast_{\varphi}((X\lrcorner \varphi)\wedge (Y\lrcorner \varphi)\wedge\gamma),
\end{align}
where $X,Y$ are tangent vectors on $M$.  Then $\rj_{\varphi}$ is surjective with kernel equal to $\Omega^3_7$ (see e.g.~\cite{Karigiannis2007}*{Proposition 2.17}). We can also define an injective linear operator $\ri_{\varphi}:S^2\to\Omega^3_1\oplus\Omega^3_{27}$ as in \cite{Bryant2006} which is (up to scaling) a right inverse for $\rj_{\varphi}$, and is given locally using summation notation by
	\begin{align}
	   	\ri_{\varphi}(h)&=\frac{1}{2}h_i^l\varphi_{ljk}e^{ijk}
	   	,\label{Eq:i.operator} 
	\end{align}
	where  $h\in S^2$ is given locally by $h_{ij}e^ie^j$.  We note that $\ri_{\varphi}(g)=3\varphi$, $\rj_{\varphi}(\varphi)=6g$ and
	\begin{equation}\label{eq:i.j}
	  \rj_{\varphi}\circ\ri_{\varphi}(h)=2(\tr h)g+4h,  
	\end{equation}
	where $\tr h=g^{ij}h_{ij}$ is the trace of $h$ with respect to $g$.  For later use, we let $S^2_0$ denote the trace-free symmetric 2-tensors on $M$ and note that $\ri_{\varphi}:S^2_0\to\Omega^3_{27}$ is an isomorphism.

Given any $\gt$-structure $\varphi$, there exist unique differential forms $\tau_0\in\Omega^0$, $\tau_1\in\Omega^1$,  $\tau_2\in\Omega_{14}^2$ and $\tau_3\in\Omega^3_{27}$, using the decomposition of forms in \eqref{eq:form.decomp}, such that  (see e.g.~\cite{Bryant2006}*{Proposition 1})
\begin{align}
  	\rd\varphi 
  	&= \tau_0\psi+3\tau_1\wedge\varphi+\ast\tau_3,
	\label{Eq:Fernandez d} \\
	\rd\psi 
	&= 4\tau_1\wedge\psi+\tau_2\wedge\varphi.
	\label{eq: Fernandez dpsi}  
\end{align}
Together, the forms $\lbrace \tau_0,\tau_1,\tau_2,\tau_3 \rbrace$ are called the \emph{intrinsic torsion forms} of the $\gt$-structure $\varphi$.
The  \emph{full torsion tensor}  $T$ is defined locally by the formula
\begin{equation}
\label{Eq.nabla.varphi}
	\nabla_i\varphi_{jkl}
	=T_i^m\psi_{mjkl} 
\end{equation}
and may be expressed using the torsion forms as (see e.g.~\cite{Karigiannis2007}*{Theorem 2.27})
\begin{equation}
\label{Eq:Torsion}
	T =\frac{\tau_0}{4}g -\tau_1^{\sharp}\lrcorner\varphi -\frac{1}{2}\tau_2 -\frac{1}{4}\rj_{\varphi}(\tau_3).
\end{equation}

By \eqref{eq: Fernandez dpsi}, a $\gt$-structure is coclosed if and only if $\tau_1=0$ and $\tau_2=0$.  Hence, the full torsion tensor of a coclosed $\gt$-structure simplifies to the symmetric $2$-tensor
\begin{equation}
\label{eq:torsion.coclosed}
    T=\frac{\tau_0}{4}g -\frac{1}{4}\rj_{\varphi}(\tau_3) 
    \in S^2.    
\end{equation}

\subsection{Contact Calabi--Yau manifolds}
\label{sec: cCY manifolds}

In this section we review the properties of  contact Calabi--Yau (cCY)  7-manifolds and the natural coclosed $\gt$-structures on them. The original sources for cCY manifolds are  \cites{Tomassini2008,Habib2015}; for a more detailed exposition and study of $\gt$-geometry on cCY $7$-manifolds, see \cite{Calvo-Andrade2020}.
	
\begin{definition}
\label{dfn:cCY}
	A \emph{contact Calabi--Yau} (cCY) manifold is a quadruple   $(M^{2n+1},g,\eta,\Upsilon)$  such that
\begin{itemize}
	\item $(M,g)$ is a $(2n+1)$-dimensional Sasakian manifold with contact form $\eta$;
	
	\item\textcolor{black}{$\Upsilon$ is a  nowhere vanishing closed transversal $(n,0)$-form on $\cD=\ker\eta$,  with $\omega=\rd\eta$ and}
	\textcolor{black}{
	$$ \frac{\omega^n}{n!} = (-1)^{\frac{n(n-1)}{2}} \Big( \frac{i}{2} \Big)^n \Upsilon \wedge \overline{\Upsilon},
	\quad \rd\Upsilon=0.
	$$}
  We shall write:
	$$ 
	\Re\Upsilon:= \frac{\Upsilon+\oep}{2},\quad \Im\Upsilon:= \frac{\Upsilon-\oep}{2\ri}.
	$$  	
	\end{itemize}
\end{definition}
	
\begin{remark}
	A contact Calabi--Yau manifold $(M,g,\eta,\Upsilon)$ has a transverse Calabi--Yau geometry on $\cD=\ker\eta$, given by $g|_{\cD}$, $\omega$ and $\Upsilon$. When the Sasakian geometry is regular or quasi-regular, $M$ is an $S^1$-(orbi)bundle over a Calabi--Yau orbifold, but the Sasakian geometry can also be irregular, and then there is no $S^1$-fibration structure on $M$ compatible with the contact Calabi--Yau geometry.
\end{remark}

We now see how to relate the contact Calabi--Yau geometry in 7 dimensions to $\gt$ geometry (cf.~\cite{Habib2015}*{Corollary 6.8} and \cite{lotay2021}). 
\begin{prop}
\label{CcY G2-structure}
	Let $(M^7,g,\eta,\Upsilon)$ be a contact Calabi-Yau 7-manifold. Then $M$ carries a \textcolor{black}{ 1-parameter family} of coclosed $\gt$-structures defined by \eqref{eq:std.phi}, 
	for $\epsilon>0$, where $\omega=\rd\eta$. Furthermore,  $\rd\varphi=\varepsilon\omega\wedge\omega$ and its corresponding dual $4$-form is given by
	\eqref{eq:std.psi}.
\end{prop}
	
Notice that $\psi$ is manifestly closed, as $\omega$ and $\Upsilon$ are closed and $\omega\wedge\Upsilon=0$ by bidegree considerations, i.e.~$\varphi$ is coclosed, and its torsion is encoded in
$
\rd\varphi=\epsilon\omega\wedge\omega.
$ Various other properties of the $\epsilon$-family \eqref{eq:std.phi} are studied in \cite{lotay2021}.

We now provide a concrete means for finding examples of contact Calabi--Yau 7-manifolds, which are circle (orbi)bundles over Calabi--Yau 3-orbifolds.

\begin{example}[Calabi--Yau links]
\label{ex: CY links}
    Given $w=(w_0,\dots,w_4)\in\Q^{5}$, the zero set of a $w$-weighted homogeneous polynomial $f\in\C[z_{0}, \dots, z_{4}]$ of degree $d=\sum_{i=0}^4 w_i$ is an affine hypersurface in $\C^5$ with an isolated singularity at $0$. 
\begin{multicols}{2}
    Its link $M_f$ on a \textcolor{black}{(sufficiently small)} $9$-sphere is a compact and $2$-connected smooth cCY $7$-manifold, fibering by circles over a Calabi--Yau $3$-orbifold $Z\subseteq \P^4(w)$ by the weighted Hopf fibration \cite{Calvo-Andrade2020}*{Theorem 1.1}:
\columnbreak
    \[\begin{tikzcd}        
    M_{f}^7 \arrow[hook]{r} \arrow{d}& S^{9} \arrow{d}\\
        Z \arrow[hook]{r}& \P^{4}(w)
    \end{tikzcd}\]
\end{multicols}
In particular, $Z$ can be taken to be any of the weighted Calabi--Yau $3$-folds listed in \cite{Candelas1990}. For a detailed survey on Calabi--Yau links, see \cite{Calvo-Andrade2020}*{\S2}.
The $\C$-family of Fermat quintics yields the simplest examples, and indeed the only ones for which the base $Z$ is smooth.
\end{example}

\begin{remark}
\label{rem: optimist topology}
    Regarding the possibility of a torsion-free $\gt$-structure on a Calabi--Yau link, we observe a number of favourable topological circumstances, in the terms of \cite{Joyce2000}*{\S10.2}. 
    
    Since it admits natural $\gt$-structures such as in \eqref{eq: standard setup}, every cCY $7$-manifold is obviously orientable and spin.
    In particular, CY links are $2$-connected, so trivially $\pi_1(M_f)$ is finite and there is no obstruction coming from an intersection form on $H^2(M_f)$. Moreover, CY links are compact, so if a torsion-free $\gt$-structure exists at all, then its induced Riemannian metric will have holonomy group precisely $\gt$.
    Finally, since every such CY link $M_f\to Z\subseteq \P^4(w)$ admits solutions to the heterotic Bianchi identity \cite{lotay2021}*{Theorem 1}, 
    its first Pontryagin class $p_1(M_f)$ coincides with the pullback of the second Chern class $c_2(Z)$ of the weighted projective $3$-orbifold, which is certainly not trivial in general, e.g.~for the Fermat quintic itself.
\end{remark}

\section{The Laplacian coflow solution}
\label{Section.Laplacian.coflow}
	
In this section we solve the Laplacian coflow \eqref{eq: Laplacian.coflow} of $\gt$-structures on a contact Calabi--Yau 7-manifold by an explicit ansatz, and we study the behaviour of the metric and torsion along the flow.

\subsection{Solving the flow}

We want to consider the Laplacian coflow starting at the natural coclosed $\gt$-structure defined by \eqref{eq:std.phi} and \eqref{eq:std.psi} on a cCY setup \eqref{eq: standard setup}:
\begin{equation}
\label{eq: phi0 and psi0} 
    \varphi_0=\epsilon\eta_0\wedge\omega_0+\Re\Upsilon_0\qandq
	\psi_0=\frac{1}{2}\omega_0^2-\epsilon\eta_0\wedge\Im\Upsilon_0.
\end{equation}

To this end, we consider the family of $\gt$-structures given by
\begin{equation}
\label{Eq.fluxo1}
	\varphi_t = f_th_t^2\eta_0\wedge \omega_0+h_t^3\Re \Upsilon_0,
\end{equation}
	for functions $f_t,h_t$ depending only on time, with 
\begin{equation}
\label{eq:init.conds.fh}
	f_0=\epsilon
	\qandq h_0=1.
\end{equation}
The induced metric and associated volume form are then given by 
\begin{equation}
\label{eq: metric.vol.t}
    g_t=f_t^2\eta_0^2+h_t^2g_{\cD_0}
    \qandq
    \vol_t=f_th_t^6\eta_0\wedge\vol_{\cD_0},
\end{equation}
with  \textcolor{black}{
\begin{equation}
\label{eq:vol.D}	
	\vol_{\cD_0}=\frac{1}{3!}\omega_0^3=\frac{i}{8}\Upsilon_0\wedge\oep_0=\frac{1}{4}\Re \Upsilon_0\wedge\Im\oep_0.
	\end{equation}}
It follows from \eqref{Eq.fluxo1}, \eqref{eq: metric.vol.t}, and \eqref{eq:vol.D} that
\begin{equation}
\label{eq.fluxo2}
	\psi_t =\frac{1}{2}h_t^4\omega_0^2-f_th_t^3\eta_0\wedge\Im\Upsilon_0.
\end{equation}
We observe that \eqref{eq: phi0 and psi0} is indeed the instance at $t=0$ of \eqref{Eq.fluxo1} and \eqref{eq.fluxo2}. 	
	 
\begin{lemma}
\label{lem:torsion.t}
    Let $\varphi_t$ be a $\gt$-structure as in \eqref{Eq.fluxo1}.  Then we have
	\begin{equation}
	\label{eq:star.d.phi.t}
	    \rd\varphi_t=f_th_t^2\omega_0^2,\qquad\rd\psi_t=0\qandq  \ast_t\rd\varphi_t=2f_t^2\eta_0\wedge\omega_0.
	\end{equation}
	Hence, the torsion forms of $\varphi_t$ as in \eqref{Eq:Fernandez d} are 
	\begin{equation}
	\label{eq:torsion.t}
	    (\tau_0)_t= \frac{6f_t}{7h_t^2},\quad (\tau_1)_t=0,\quad (\tau_2)_t=0,\quad (\tau_3)_t=\frac{8}{7}f_t^2\eta_0\wedge\omega_0-\frac{6}{7}f_th_t\Re\Upsilon_0.	    
	\end{equation}
\end{lemma}

\begin{proof}
From \eqref{Eq.fluxo1}, \eqref{eq.fluxo2}, and Definition \ref{dfn:cCY}, we easily see that
	\begin{equation}\label{eq:d.phi.t}
	    \rd\varphi_t=f_th_t^2\omega_0^2\qandq \rd\psi_t=0,
	\end{equation}
	since $\rd\eta_0=\omega_0$ and $\omega_0\wedge\Upsilon_0=0$.  	 We also note from \eqref{eq: metric.vol.t} and \eqref{eq:d.phi.t} that
\begin{equation*} 
    \begin{split}
	 \ast_t\rd\varphi_t &=\ast_t(f_th_t^2\omega_0\wedge\omega_0)=\ast_t(2f_th_t^{-2}\cdot \textstyle\frac{1}{2}h_t^4\omega_0^2) 	=2f_th_t^{-2}\cdot f_th_t^2\eta_0\wedge\omega_0\\
&= 2f_t^2\eta_0\wedge\omega_0,
\end{split}
	\end{equation*}
	which then yields \eqref{eq:star.d.phi.t}.
	
We can now compute the torsion forms of $\varphi_t$  as follows. First,
\begin{equation}
\label{Eq:Tosion.t1}
	    \begin{split}
	        	(\tau_0)_t &= \frac{1}{7}\ast_t(\varphi_t\wedge\rd\varphi_t)= \frac{1}{7}\ast_t\big((f_th_ t^2\eta_0\wedge\omega_0+h_t^3\Re \Upsilon_0)\wedge f_th_t^2\omega_0^2\big)\\
	&=\frac{1}{7}\ast_t f_t^2h_t^4\eta_0\wedge\omega_0^3= \frac{6f_t}{7h_t^2}.
	    \end{split}
\end{equation}
	We know that $(\tau_1)_t=0$ and $(\tau_2)_t=0$ since $\varphi_t$ is coclosed.
	Furthermore, from \eqref{eq:d.phi.t} and \eqref{Eq:Tosion.t1}, we have
\begin{equation*}
	    \begin{split}
	      (\tau_3)_t&= \ast_t\rd\varphi_t-(\tau_0)_t\varphi_t\\
	&= 2f_t^2\eta_0\wedge\omega_0-\frac{6f_t}{7h_t^2}(f_th_t^2\eta_0\wedge\omega_0+h_t^3\Re\Upsilon_0)\\
	&= \frac{8}{7}f_t^2\eta_0\wedge\omega_0-\frac{6}{7}f_th_t\Re\Upsilon_0.
 	    \end{split}
\end{equation*}
Equation \eqref{eq:torsion.t} then follows.
\end{proof}

We can now prove Theorem \ref{thm:Laplacian.coflow.cCY}, by finding $f_t,h_t$ so that \eqref{eq.fluxo2} is a solution of the Laplacian coflow \eqref{eq: Laplacian.coflow}. 

	\begin{proof}[Proof of Theorem \ref{thm:Laplacian.coflow.cCY}]
	On the cCY setup \eqref{eq: standard setup}, consider the family of $\gt$-structures given by \eqref{Eq.fluxo1} and \eqref{eq.fluxo2}. We may then compute the Laplacian of $\psi_t$ using \eqref{eq:d.phi.t}: 
\begin{equation}
\label{Laplaciano1}
\begin{split}
    \Delta_t\psi_t & =  \rd\ast_t\rd\varphi_t
	= \rd (2f_t^2\eta_0\wedge\omega_0)=2f_t^2\omega_0\wedge\omega_0.     
	    \end{split}
\end{equation}
	Differentiating \eqref{eq.fluxo2} with respect to $t$ and using \eqref{Laplaciano1}, we can compute $\frac{\partial \psi}{\partial t}=\Delta_t\psi_t$ and then equate the coefficients of $\eta_0\wedge\Im\Upsilon_0$ and  $\omega_0^2$ (since the flow preserves the class of \eqref{eq.fluxo2}) to obtain
	\textcolor{black}{
	\begin{equation}\label{eq.Laplacian1}
	\frac{d}{d t}(h_t^4)=4f_t^2,\quad\frac{d}{d t }(f_th_t^3)=0.
	\end{equation}}
	Therefore, from   \eqref{eq:init.conds.fh} and the second equation in \eqref{eq.Laplacian1} we conclude that
	\begin{equation}\label{eq:fh-3}
	f_t=\varepsilon h_t^{-3}.
	\end{equation}
	Substituting \eqref{eq:fh-3} into \eqref{eq.Laplacian1}, we have  
	\begin{equation}\label{eq:h.ode}
	\frac{d}{d t}(h_t^4)=4\epsilon^2 h_t^{-6}.
	\end{equation}
	The ODE \eqref{eq:h.ode} can be easily solved and so, together with \eqref{eq:init.conds.fh} and \eqref{eq:fh-3}, we find that
	\begin{align}\label{Eq: solution.ODE}
 f_t = \epsilon(1+10\epsilon^2 t)^{-3/10}\qandq 	h_t = (1+10\epsilon^2 t)^{1/10}. 
	\end{align}
	
	In conclusion, we have found a solution to the Laplacian coflow \eqref{eq: Laplacian.coflow} in the cCY setup \eqref{eq: standard setup}, with initial condition \eqref{eq: phi0 and psi0}, which induces the following family of $\gt$-structures and their duals, metrics and volume forms:
\begin{align}
     \varphi_t&=\epsilon(1+10\epsilon^2t)^{-1/10}\eta_0\wedge\omega_0+(1+10\epsilon^2t)^{3/10}\Re\Upsilon_0;
\label{eq:simple.sol.1}\\
	\psi_t&=\frac{1}{2}(1+10\epsilon^2t)^{2/5}\omega_0^2-\epsilon\eta_0\wedge\Im\Upsilon_0;
\label{eq:simple.sol.2}\\
	g_t&=\epsilon^2(1+10\epsilon^2t)^{-3/5}\eta_0^2+(1+10\epsilon^2t)^{1/5}g_{\mathcal{D}_0};
\label{eq:simple.sol.3}\\
	\vol_t&=\epsilon(1+10\epsilon^2t)^{3/10}\eta_0\wedge\vol_{\mathcal{D}_0},
\label{eq:simple.sol.4}
\end{align}
	where $\mathcal{D}_0=\ker\eta_0$, defined for all $t\in(-\frac{1}{10\epsilon^2},\infty)$.
\end{proof}

\begin{remark}	We notice that the solution to the Laplacian coflow we have obtained is \emph{immortal} (i.e.~exists for all positive time), but it is not eternal, since it fails to exist for $t\leq -\frac{1}{10\epsilon^2}$.  
\textcolor{black}{This will be the unique solution} of the form \eqref{Eq.fluxo1} satisfying \eqref{eq:init.conds.fh}, and we expect that it will be the unique solution starting at \eqref{eq: phi0 and psi0} in general, but such a general theory is lacking as mentioned in $\S$\ref{ss:L.coflow}.
\end{remark}
	
\begin{remark}	
	We see from \eqref{eq:simple.sol.4} that the volume form is strictly increasing pointwise in time.  We therefore confirm that, if $M$ is compact, then its volume is indeed strictly increasing in time, tending to infinity:
\begin{equation*}
	\text{Vol}(M,g_t) \to\infty\quad\text{as }t\to\infty.
\end{equation*}
We also notice that $\psi_t$ does indeed stay in the cohomology class $[\psi_0]$, \textcolor{black}{since $\omega_0^2$ is exact on $M$.}
	
	However, suppose the contact Calabi--Yau structure is quasi-regular, so that $M$ is an $S^1$-(orbi)bundle over a Calabi--Yau 3-orbifold $Z$. We then observe that, by \eqref{Laplaciano1},  
	$$
	\Delta_t\psi_t=2f_t^2\omega_0\wedge\omega_0,
	$$
	which is a transverse 4-form that will descend to $Z$\color{black}{ and cannot be exact on $Z$, since otherwise (as $\omega_0$ is closed) the volume form $\omega_0^3$ would be exact on the compact base $Z$}. \color{black}
 Hence, the Laplacian coflow does not induce a flow of $\SU(3)$-structures on $Z$ with forms staying in some fixed cohomology class.
	\end{remark}

	\subsection{Metric and curvature}
	It is worth studying aspects of the asymptotic behaviour of the solution to the Laplacian coflow we have found in terms of the metric geometry.  
	
	It is straightforward to see from \eqref{eq:simple.sol.3} that, as $t\to\infty$, the direction dual to $\eta_0$ on $M$ collapses whilst the transverse geometry given by $\cD_0$ expands.  On the other hand, the reverse situation occurs as $t\to-\frac{1}{10\epsilon^2}$.  
	
	\begin{example}
	For illustration, let us suppose that $M$ is compact and that the cCY structure is quasi-regular, so that $M$ is an $S^1$-(orbi)bundle over a compact Calabi--Yau 3-orbifold $Z$ with metric $g_Z$.  Then, as $t\to\infty$ the circle fibres collapse, whilst the base $Z$ expands to become $\mathbb{C}^3$ with the flat metric, since the coefficient of $g_{\mathcal{D}_0}=g_Z$ in \eqref{eq:simple.sol.3} tends to infinity.  As $t\to -\frac{1}{10\epsilon^2}$, on the other hand, the base $Z$ collapses to a point and the circles expand to become the real line $\mathbb{R}$.
	\end{example}
	
In the analysis so far we have neglected the fact that the Laplacian coflow is the gradient flow of the volume functional, and so the volume must always be increasing along the flow.  To remedy this, we may rescale the family $(M,g_t)$ so that the volume is fixed and obtain the following.  

\begin{lemma}
\label{lemma: LapCoflow_NormalisedVol}
	Let $M$ be compact. After normalising $(M,g_t)$ to a fixed volume,  the Laplacian coflow solution collapses to $\mathbb{R}$, as $t\to-\frac{1}{10\epsilon^2}$, and to $\mathbb{C}^3$ with the flat metric, as $t\to\infty$.
\end{lemma}
\begin{proof}
	Fixing the volume to be constant is equivalent to multiplying the metric by $(1+10\epsilon^2t)^{-3/35}$ by \eqref{eq:simple.sol.4}, which gives
	$$(1+10\epsilon^2t)^{-3/35}g_t=(1+10\epsilon^2t)^{-24/35}\eta_0^2+(1+10\epsilon^2t)^{4/35}g_{\mathcal{D}_0}.$$
	The result then follows.
\end{proof}

We now take the standard approach to understanding the behaviour of volume along a family of compact manifolds, by normalising the diameter \textcolor{black}{instead}. From the formula \eqref{eq:simple.sol.3} for the metric $g_t$, we immediately obtain:
\begin{lemma}
\label{lemma: LapCoflow_NormalisedDiam}
    Let $M$ be compact. After normalising $(M,g_t)$ to unit diameter,  the Laplacian coflow solution  is volume-collapsing to a 1-dimensional space $M_-$ as $t\to-\frac{1}{10\epsilon^2}$ and to a 6-dimensional space $M_+$ as $t\to\infty$.  If the Sasakian structure is quasi-regular, so that $M$ is an $S^1$-orbibundle over a Calabi--Yau 3-orbifold $Z$, then $M_-=S^1$ and $M_+=Z$.
\end{lemma}
\color{black}\begin{proof}
From the formula \eqref{eq:simple.sol.3} for the metric $g_t$, we see that normalising $(M,g_t)$ to unit diameter amounts to considering the rescaled metric
\[
C(1+10\epsilon^2t)^{-1/5}g_t=C\epsilon^2(1+10\epsilon^2t)^{-4/5}+Cg_{\mathcal{D}_0}
\]
for some suitable constant $C>0$.  The result then follows.
\end{proof}\color{black}
In the light of what is known for the Laplacian flow of closed $\gt$-structures, cf.~\cite{Lotay2015}*{Theorem 8.1}, we might expect that a condition on the uniform continuity of the metrics along the Laplacian coflow of coclosed $\gt$-structures, together with a (pointwise) bound on the torsion tensor, would lead to long-time existence results.  Since no such general theory currently exists, it is therefore useful to examine the uniform continuity properties of the solution to the Laplacian coflow we have found.

\begin{lemma}
\label{prop:uniformly.continuous.metric} 
    Let $\varphi_t$ be the solution to the Laplacian coflow given by \eqref{eq:simple.sol.1}. Then the associated  metric $g_t$ is uniformly continuous (in $t$) on any compact interval contained in $(-\frac{1}{10\epsilon^2},\infty)$, but it is not uniformly continuous on $(-\frac{1}{10\epsilon^2},T)$ or $(T,\infty)$ for any $T$.
\end{lemma}
\begin{proof}
We immediately see from the formula \eqref{eq:simple.sol.3} for the metrics $g_t$ that, since the 2-tensors $\eta_0^2$ and $g_{\cD_0}$ are fixed, the uniform continuity in $t$ of $g_t$ on an interval $I$ is equivalent to the uniform continuity of both $(1+10\epsilon^2t)^{-3/5}$ and $(1+10\epsilon^2t)^{1/5}$ on $I$.  The result follows.
\end{proof}

We now study the behaviour of the Riemann curvature tensor along our Laplacian coflow solution, based on some computations of the Riemannian geometry of cCY $7$-manifolds carried out in detail in \cite{lotay2021}*{\S3.1}.

\begin{prop}
\label{prop:curvature.t}
    Let $\varphi_t$ be the solution to the Laplacian coflow given by \eqref{eq:simple.sol.1} with associated  metric $g_t$ as in \eqref{eq:simple.sol.3}. 
    Let $Rm_t$ denote the Riemann curvature tensor of $g_t$ and let $Rm_0^{\cD_0}$ denote the curvature of the transverse connection on $\cD_0$ induced by the Levi-Civita connection of $g_0$.  Then
\[
    |Rm_t|_{g_t}^2= (1+10\epsilon^2t)^{-2/5}|Rm_0^{\cD_0}|_{g_0}^2+c_0\epsilon^4 (1+10\epsilon^2t)^{-2}
\]
    for some constant $c_0>0$.
\end{prop}
\begin{proof}
    If we write $g_t$ as in \eqref{eq: metric.vol.t} for functions $f_t,h_t$ given in \eqref{Eq: solution.ODE}, we may observe that if we let
\begin{equation}
\label{eq:overline.g.t}
    \overline{g}_t
    =\frac{f_t^2}{h_t^2}\eta_0^2+g_{\cD_0}
\end{equation}
    then $g_t=h_t^2\overline{g}_t$.  Hence, if $\overline{Rm}_t$ is the Riemann curvature tensor of $\overline{g}_t$ in \eqref{eq:overline.g.t}, we deduce that 
\begin{equation}
\label{eq:rescaling.Rm}
    |Rm_t|_{g_t}
    =\frac{1}{h_t^2}|\overline{Rm}_t|_{\overline{g}_t}.
\end{equation}
    It is a direct consequence of \cite{lotay2021}*{Proposition 3.2} that the transverse connection on $\cD_0$ induced by $\overline{g}_t$ is the same as for $g_0$, and hence its curvatures are independent of $t$.  It follows from \cite{lotay2021}*{Proposition 3.8} that the remaining components of the Riemann curvature of $\overline{g}_t$ scale by $f_t^2/h_t^2$ in comparison to the curvature of the standard Sasakian metric $g$.  We deduce that
\begin{equation}
\label{eq:overline.Rm.t}
    |\overline{Rm}_t|^2_{\overline{g}_t}=|Rm_0^{\cD_0}|_{g_0}^2+\frac{c_0f_t^4}{h_t^4}
\end{equation}
    for some constant $c_0>0$.  The assertion now follows from \eqref{Eq: solution.ODE}, \eqref{eq:rescaling.Rm}, and \eqref{eq:overline.Rm.t}. 
\end{proof}

\begin{remark}
\label{rmk:curv.t}
    We see from Proposition \ref{prop:curvature.t} that the curvature of the metric $g_t$ given by the solution to the Laplacian coflow decays to $0$ as $t\to\infty$ but blows up as $t$ approaches the finite negative time singularity at $-\frac{1}{10\epsilon^2}$.  However, if we consider $t|Rm_t|_{g_t}$, this will tend to infinity as $t\to\infty$ unless the curvature $Rm^{\cD_0}_0$ of the transverse connection is zero.  However, we see that $(1+10\epsilon^2t)|Rm_t|_{g_t}$ will remain bounded as $t\to-\frac{1}{10\epsilon^2}$.  We will return to these observations later.
\end{remark}


\subsection{Full torsion tensor}  We now wish to analyse the behaviour of the torsion along our solution to the Laplacian coflow.  We begin by writing down the full torsion tensor for the solution.

\begin{prop}
\label{prop:full.torsion.t}
    The full torsion tensor $T_t$ of the solution to the Laplacian coflow in \eqref{eq:simple.sol.1}--\eqref{eq:simple.sol.4} is
\begin{equation}\label{eq:full.torsion.t}
\begin{split}
    T_t&=\frac{1}{4}(\tau_0)_tg_t-\frac{1}{4}\rj_{\varphi_t}\big((\tau_3)_t\big)\\
    &=\frac{3\epsilon}{14}(1+10\epsilon^2t)^{-1/2}g_t\\
    &\quad-\frac{1}{14}\rj_{\varphi_t}\left(
    4\epsilon^2(1+10\epsilon^2t)^{-3/5}\eta_0\wedge\omega_0-3\epsilon(1+10\epsilon^2t)^{-1/5}\Re\Upsilon_0\right)\\
    &=-\frac{3}{2}\epsilon^3(1+10\epsilon^2t)^{-11/10}\eta_0^2+\frac{1}{2}\epsilon(1+10\epsilon^2t)^{-3/10}g_{\cD_0}. 
    \end{split}
\end{equation}
\end{prop}

\begin{proof} This first two expressions for $T_t$ follow from \eqref{eq:torsion.coclosed}, Lemma \ref{lem:torsion.t} and \eqref{Eq: solution.ODE}.  

For the final expression, we first let $\xi_0$ be the dual vector field to $\eta_0$, we  let $X,Y\in\cD_0$ and consider $\varphi_t$ as in \eqref{Eq.fluxo1} for functions $f_t$ and $h_t$.
 We recall that, since $(\omega_0,\Upsilon_0)$ defines an $\SU(3)$-structure on $\cD_0$ we have that $\omega_0\wedge\Re\Upsilon_0=0$,
\begin{align}
(X\lrcorner\Re\Upsilon_0)\wedge(Y\lrcorner\Re\Upsilon_0)\wedge\omega_0&=2g_{\cD_0}(X,Y)\vol_{\cD_0}\label{eq:SU3.structure.1}\\
(X\lrcorner \omega_0)\wedge (Y\lrcorner \Re\Upsilon_0)\wedge\Re\Upsilon_0&=-2g_{\cD_0}(X,Y)\vol_{\cD_0}.\label{eq:SU3.structure.2}
\end{align}
\textcolor{black}{
It follows from} \eqref{Eq:j.operator}, \eqref{Eq.fluxo1}, \eqref{eq: metric.vol.t}, and \eqref{eq:SU3.structure.1} that
\begin{align}
\rj_{\varphi_t}(\eta_0\wedge\omega_0)(\xi_0,\xi_0)&=\ast_t(f_t^2h_t^4\eta_0\wedge\omega_0^3)=6\frac{f_t}{h_t^2},\label{eq:j.eta.omega.1}\\    
\rj_{\varphi_t}(\eta_0\wedge\omega_0)(\xi_0,X)&=0,\label{eq:j.eta.omega.2}\\
\rj_{\varphi_t}(\eta_0\wedge\omega_0)(X,Y)&=\ast_t(2h_t^6g_{\cD_0}(X,Y)\eta_0\wedge\vol_{\cD_0})=\frac{2}{f_t}g_{\cD_0}(X,Y).\label{eq:j.eta.omega.3}
\end{align}
Similarly, using \eqref{eq:SU3.structure.2},
\begin{align}
    \rj_{\varphi_t}(\Re\Upsilon_0)(\xi_0,\xi_0)
    &=0, \label{eq:j.ReUpsilon.1}\\
    \rj_{\varphi_t}(\Re\Upsilon_0)(\xi_0,X)
    &=0, \label{eq:j.ReUpsilon.2}\\
    \rj_{\varphi_t}(\Re\Upsilon_0)(X,Y)
    &=*_t(4f_th_t^5g_{\cD_0}(X,Y)\eta_0\wedge\vol_{\cD_0}) =\frac{4}{h_t}g_{\cD_0}(X,Y), \label{eq:j.ReUpsilon.3}
\end{align}

Applying \eqref{eq:j.eta.omega.1}--\eqref{eq:j.ReUpsilon.3} along with the particular expressions for $f_t$ and $h_t$ in \eqref{Eq: solution.ODE} then gives that
\begin{multline}\label{eq:j.t}
\rj_{\varphi_t}\left(
    4\epsilon^2(1+10\epsilon^2t)^{-3/5}\eta_0\wedge\omega_0-3\epsilon(1+10\epsilon^2t)^{-1/5}\Re\Upsilon_0\right)\\
    =24\epsilon^3(1+10\epsilon^2t)^{-11/10}\eta_0^2-4\epsilon(1+10\epsilon^2t)^{-3/10}g_{\cD_0}.
\end{multline}
Substituting \eqref{eq:j.t} into \eqref{eq:full.torsion.t} gives the claimed final expression for $T_t$.
\end{proof}

Given the description of the full torsion tensor in Proposition \ref{prop:full.torsion.t}, we can compute its norm, the norm of its gradient and its divergence as follows.

\begin{prop}
\label{prop:norm.full.torsion.t}
    Let $T_t$ be the full torsion tensor of the solution to the Laplacian coflow in \eqref{eq:simple.sol.1}--\eqref{eq:simple.sol.4}.  Then
\begin{align*}
    |T_t|_{g_t}^2&=\frac{15}{4}\epsilon^2(1+10\epsilon^2t)^{-1},\\
    |\nabla_t T_t|_{g_t}^2&=c_0\epsilon^4(1+10\epsilon^2t)^{-2},\\
    \div_tT_t&=0,
\end{align*}
    where $c_0>0$ is a constant, $\nabla_t$ is the Levi-Civita connection of $g_t$ and $\div_t$ is the divergence with respect to the metric $g_t$.
\end{prop}
\begin{proof}
\textcolor{black}{
From \eqref{eq:simple.sol.3}, we have
\begin{equation*}
    |\epsilon^2\eta_0^2|_{g_t}^2=(1+10\epsilon^2t)^{6/5}\qandq |g_{\cD_0}|^2_{g_t}=6(1+10\epsilon^2t)^{-2/5}.
\end{equation*}
We then deduce the first claimed equation:
\begin{equation*}
    |T_t|_{g_t}^2=\frac{9}{4}\epsilon^2(1+10\epsilon^2t)^{-1}+\frac{3}{2}\epsilon^2(1+10\epsilon^2t)^{-1}=\frac{15}{4}\epsilon^2(1+10\epsilon^2t)^{-1}.
\end{equation*}
On the other hand, it follows from \eqref{eq:simple.sol.3} and Proposition \ref{prop:full.torsion.t} that
\begin{equation}
\label{eq:Tt.decomp}
    T_t=-2\epsilon^3(1+10\epsilon^2t)^{-11/10}\eta_0^2+\frac{1}{2}\epsilon(1+10\epsilon^2t)^{-1/2}g_t.
\end{equation}
}
\textcolor{black}{As to the remaining claimed equations, notice that the vector field $\xi_0$, dual to $\eta_0$, generates a local symmetry of the $\gt$-structure $\varphi_t$, for each $t$, hence also of the metric $g_t$, and $| \xi_0 |_{g_t}$ is constant. In summary, 
$$
\nabla_tg_t=0
\qandq
(\nabla_t)_{\xi_0}\xi_0=0.
$$}
We deduce from \eqref{eq:Tt.decomp} that $\div_tT_t=0$ as claimed and that the only non-zero terms in $\nabla_tT_t$ arise from $(\nabla_t)_X\eta_0^2$ for $X\in\cD_0$.


\color{black}
Since $g_t$ is conformal to $\overline{g}_t$ in \label{eq:overline.g.t}  (with conformal factor $h_t^2$ which only depends on $t$), we may use \cite{lotay2021}*{Proposition 3.2} which gives the connection matrix for the Levi-Civita connection of $\overline{g}_t$, and hence $g_t$ (where $\epsilon=f_t/h_t$ in the notation of \cite{lotay2021}*{Proposition 3.2}).  If we let $\{e_1,Je_1,e_2,Je_2,e_3,Je_3\}$ denote a local tranverse $\mathrm{SU}(3)$ coframe, so that $g_{\mathcal{D}_0}=\sum_{j=1}^3 e_j^2+(Je_j)^2$, we see explicitly that we have the local expression
\[
\nabla_{t}\left(\frac{f_t}{h_t}\eta_0\right)=\frac{f_t}{2h_t}\sum_{j=1}^3 e_j\otimes Je_j-Je_j\otimes e_j.
\]
Therefore, if $X\in\cD_0$ such that $g_t(X,X)=1$, we  have $\big((\nabla_t)_X\eta_0\big)^{\sharp}\in\cD_0$, 
and
\begin{equation*}
    |(\nabla_t)_X(f_t^2\eta_0^2)|_{g_t}^2=\frac{c_0f_t^2}{h_t^4},
\end{equation*} 
for some constant $c_0>0$.  Here we used the fact that  
$g_t(Y,Y)=h_t^2g_{\cD_0}(Y,Y)$ for any $Y\in\cD_0$. \color{black}  Choosing $f_t$ and $h_t$ as in \eqref{Eq: solution.ODE}, we finally obtain
\begin{align*}
    |(\nabla_t)_XT_t|^2_{g_t}&=|2\epsilon(1+10\epsilon^2t)^{-1/2}(\nabla_t)_X(\epsilon^2(1+10\epsilon^2t)^{-3/5}\eta_0^2)|^2_{g_t}\\
    &=4\epsilon^2(1+10\epsilon^2t)^{-1}\cdot c_0\epsilon^2(1+10\epsilon^2t)^{-1}\\
    &=4c_0\epsilon^4(1+10\epsilon^2t)^{-2}.
    \qedhere
\end{align*}
\end{proof}

\begin{remark}
We see from Proposition \ref{prop:norm.full.torsion.t} that $|T_t|_{g_t}^2$ and $|\nabla_tT_t|_{g_t}$ have the same dependence on $t$, as we would expect from the general theory in the study of flows of $\gt$-structures, and that there is a component of $|Rm_t|_{g_t}$ with this same dependence on $t$, by Proposition \ref{prop:curvature.t}.  However, there is a component of $|Rm_t|_{g_t}$ which has a different dependence on $t$, as long as the tranverse geometry on the contact Calabi--Yau is not flat: this is unsurprising because, in particular, having small torsion does not imply that the metric is close to flat.
\end{remark}

\subsection{Singularity analysis}

We can now combine all of the estimates we have obtained to describe the infinite time singularity for the Laplacian coflow in terms of the types in Definition \ref{dfn:sing.types}, which proves Theorem \ref{thm:sings}(a).

\begin{prop}
In a cCY setup \eqref{eq: standard setup}, suppose moreover $M^7$ is compact, and let $K:=\sup_M|Rm_0^{\cD_0}|_{g_0}$. Then there is a constant $c_0>0$, independent of $\epsilon$, such that the quantity $\Lambda(t)$ in \eqref{eq:Lambda.t}, along the Laplacian coflow solution given by Theorem \ref{thm:Laplacian.coflow.cCY}, is given by
\[
\Lambda(t)=\big(K^2(1+10\epsilon^2t)^{-2/5}+c_0\epsilon^4(1+10\epsilon^2t)^{-2}\big)^{1/2}.
\]
Hence, the Laplacian coflow has a Type IIb infinite time singularity, unless $g_{\cD_0}$ is flat, in which case it has a Type III infinite time singularity.
\end{prop}

\begin{proof}
The formula for $\Lambda(t)$ follows immediately from Proposition \ref{prop:curvature.t} and Proposition \ref{prop:norm.full.torsion.t}.  We then see that 
$t\Lambda(t)\to\infty$ as $t\to\infty$ unless $K=0$, from which the result then follows by Definition \ref{dfn:sing.types}.
\end{proof}

\begin{remark}
As remarked earlier, these are the first examples of Laplacian coflow solutions which have Type IIb singularities.  The other examples are all homogeneous, arising from nilmanifold and solvmanifold constructions, and in those cases one always has a Type III singularity if it does not converge to a flat torsion-free $\rG_2$-structure.
\end{remark}

\section{The Laplacian flow solution}
\label{Section.Laplacian.flow}

We now want to consider the Laplacian flow \eqref{eq: Laplacian.flow} 
starting at one of the natural coclosed $\gt$-structures on a contact Calabi--Yau 7-manifold. Whilst the class of coclosed $\gt$-structures is not expected to be preserved, in general, along the Laplacian flow, we will see that it is in our setting and that we can solve the flow explicitly.  By again studying the evolution of the metric and torsion along an ansatz for the flow, we will see that the behaviour is rather different from the Laplacian coflow, having a finite time singularity but existing for all negative times.

\subsection{Solving the flow}

Again, we choose the natural coclosed $\gt$-structure $\varphi_0$ on a cCY setup \eqref{eq: standard setup}, cf.~\eqref{eq: phi0 and psi0}, as the initial condition for our flow.  
We are therefore again led to consider the $\gt$-structures $\varphi_t$ in \eqref{Eq.fluxo1} depending on functions $f_t$, $h_t$ (which only depend on time) satisfying the initial conditions \eqref{eq:init.conds.fh} to prove Theorem \ref{thm:Laplacian.flow.cCY} as follows.

\begin{proof}[Proof of Theorem \ref{thm:Laplacian.flow.cCY}]
    We compute the Laplacian of $\varphi_t$ using \eqref{eq: metric.vol.t} and  \eqref{Laplaciano1}:
\begin{equation}
\label{eq:Laplacian.cCY}
\begin{split}
    \Delta_{\varphi_t}\varphi_t
    &=\ast_{\varphi_t}(\Delta_{\psi_t}\psi_t) =\ast_{\varphi_t}(2f_t^2\omega_0^2) =\ast_{\varphi_t}\left(4\frac{f_t^2}{h_t^4}\cdot\frac{h_t^4\omega_0^2}{2}\right) =4\frac{f_t^2}{h_t^4}\cdot f_th_t^2\eta_0\wedge\omega_0\\
    &=\frac{4f_t^3}{h_t^2}\eta_0\wedge\omega_0.
    \end{split}
\end{equation}
    We deduce from \eqref{Eq.fluxo1} and \eqref{eq:Laplacian.cCY} that the Laplacian flow \eqref{eq: Laplacian.flow} preserves the class of coclosed $\gt$-structures in \eqref{Eq.fluxo1}.  Moreover, the Laplacian flow is equivalent to the following pair of ODEs:
 \begin{equation}
 \label{eq:Lflow.ODEs}
    \frac{d}{dt}(f_th_t^2) =\frac{4f_t^3}{h_t^2},\quad \frac{d}{dt}(h_t^3) =0.
 \end{equation}
    We deduce from the second equation in \eqref{eq:Lflow.ODEs} and \eqref{eq:init.conds.fh} that $h_t=1$.  We therefore see that the first equation in \eqref{eq:Lflow.ODEs} becomes
\begin{equation}
\label{eq:Lflow.ODE.f}
    \frac{d}{dt}f_t =4f_t^3.
\end{equation}
    From \eqref{eq:Lflow.ODE.f} and \eqref{eq:init.conds.fh} we conclude that
\begin{equation}
\label{eq:Lflow.ODE.sol}
    f_t=\epsilon(1-8\epsilon^2t)^{-1/2}\qandq h_t=1.
\end{equation}

We conclude from \eqref{Eq.fluxo1}, \eqref{eq: metric.vol.t}, and \eqref{eq.fluxo2} that we have a solution $\varphi_t$ to the Laplacian flow starting at $\varphi_0$ in \eqref{eq: phi0 and psi0} which induces the following data:
\begin{align}
    \varphi_t&=\epsilon(1-8\epsilon^2t)^{-1/2}\eta_0\wedge\omega_0+\Re\Upsilon_0,\label{eq:Lflow.1}\\
    \psi_t&=\frac{1}{2}\omega_0^2-\epsilon(1-8\epsilon^2t)^{-1/2}\eta_0\wedge\Im\Upsilon_0,\label{eq:Lflow.2}\\
    g_t&=\epsilon^2(1-8\epsilon^2t)^{-1}\eta_0^2+g_{\cD_0},\label{eq:Lflow.3}\\
    \vol_t&=\epsilon(1-8\epsilon^2t)^{-1/2}\eta_0\wedge\vol_{\cD_0},\label{eq:Lflow.4}
\end{align}
for $t\in(-\infty,\frac{1}{8\epsilon^2})$.
\end{proof}

\begin{remark}
Our solution to the Laplacian flow is \emph{ancient} (i.e.~exists for all negative times), but not eternal, forming a finite-time singularity at $t=\frac{1}{8\epsilon^2}>0$.  Again, we cannot guarantee that this is the unique solution to the Laplacian flow starting at \eqref{eq: phi0 and psi0}, even if we assume $M$ is compact, because we do not have general analytic theory for the Laplacian flow when not restricted to closed $\gt$-structures.
\end{remark}

\begin{remark}
Even though the Laplacian flow no longer has the interpretation as the gradient flow of the volume functional in this setting, we observe from \eqref{eq:Lflow.4} that the volume form of $M$ in the metric $g_t$ is still pointwise strictly increasing in time, just as for the Laplacian coflow.  Moreover, if $M$ is compact, we conclude that the volume is strictly increasing in time, but now tending to infinity in \emph{finite} time:
\[
\text{Vol}(M,g_t)\to\infty\quad \text{as $t\to\frac{1}{8\epsilon^2}$.}
\]
We also notice from \eqref{eq:Lflow.3} that the cohomology class of $\psi_t$ satisfies
\[
[\psi_t]=(1-8\epsilon^2t)^{-1/2}[\psi_0]
\]
and so it is not constant, unless $\psi_0$ is exact.  Moreover, the cohomology class $[\psi_t]$ degenerates precisely at the singular time $t=\frac{1}{8\epsilon^2}$, whenever $[\psi_0]\neq 0$.
\end{remark}

\subsection{Metric and curvature}

Here, we see from \eqref{eq:Lflow.3} that the transverse geometry stays constant along the Laplacian flow, and the direction dual to $\eta_0$ in $M$ expands as $t$ increases to $\frac{1}{8\epsilon^2}$ and collapses as $t\to-\infty$.

\begin{example}
Let $M^7$ be a compact quasi-regular cCY 7-manifold, so that it is a circle bundle over a compact Calabi--Yau 3-orbifold $Z$ with metric $g_Z$.  Then, for our Laplacian flow solution, we see from \eqref{eq:Lflow.3} that the induced metric on $Z$ is fixed along the flow.  However, the circle fibres expand to infinite size (and so become copies of $\mathbb{R}$) as $t\to\frac{1}{8\epsilon^2}$, and that $M$ collapses to $Z$ as $t\to-\infty$.  
\end{example}

\textcolor{black}{In the same spirit of Lemma \ref{lemma: LapCoflow_NormalisedVol}, for}  the Laplacian coflow, we can account for the volume increasing along the flow by normalising the volume of the evolving metrics to be the same as for $g_0$, yielding 
\[
(1-8\epsilon^2t)^{1/7}g_t=\epsilon^2(1-8\epsilon^2t)^{-6/7}\eta_0^2+(1-8\epsilon^2t)^{1/7}g_{Z}.
\]
This immediately gives the following result.

\begin{lemma}
    Let $M$ be compact. After normalising the volume of $(M,g_t)$ to be fixed, the Laplacian flow solution collapses to $\mathbb{R}$ as $t\to\frac{1}{8\epsilon^2}$ and collapses to $\mathbb{C}^3$ with the flat metric as $t\to-\infty$.
\end{lemma}
\textcolor{black}{On the other hand, as also done above in Lemma \ref{lemma: LapCoflow_NormalisedDiam},} we can normalise the diameter of the metric \eqref{eq:Lflow.3}, immediately obtaining:

\begin{lemma}
    Let $M$ be compact. Then, after normalising $(M,g_t)$ to unit diameter, the Laplacian flow  is volume-collapsing to a $6$-dimensional space $M_-$ as $t\to-\infty$ and to a $1$-dimensional space $M_+$ as $t\to\frac{1}{8\epsilon^2}$.  If the Sasakian structure is quasi-regular, so that $M$ is an $S^1$-orbibundle over a Calabi--Yau 3-orbifold $Z$, then $M_-=Z$ and $M_+=S^1$.
\end{lemma}

We see that there is a finite-time singularity for our flow at $t=\frac{1}{8\epsilon^2}$. 
As we know from the study of the Laplacian flow for closed $\gt$-structures, the formation of singularities should be related to the lack of uniform continuity of the evolving metrics and the blow-up of curvature, so we examine these properties for our solution.

\begin{prop}
Let $\varphi_t$ be the solution to the Laplacian flow given by \eqref{eq:Lflow.1} with associated metric $g_t$  as in \eqref{eq:Lflow.3}.  Then $g_t$ is uniformly continuous (in $t$) on $(-\infty,T]$, for any $T<\frac{1}{8\epsilon^2}$, but it is \emph{not} uniformly continuous on $[T,\frac{1}{8\epsilon^2})$, for any $T$.
\end{prop}

\begin{proof}
This is immediate from the formula \eqref{eq:Lflow.3} for the evolving metric $g_t$, which shows that its uniform continuity on an interval $I$ is equivalent to the uniform continuity of $(1-8\epsilon^2t)^{-1}$ on $I$.
\end{proof}

Moreover, one can immediately express the norm of curvature along the flow, from the proof of Proposition \ref{prop:curvature.t} and \eqref{eq:Lflow.ODE.sol}, as follows.
\begin{prop}
\label{prop:curvature.t.2}
    Let $\varphi_t$ be the solution to the Laplacian flow given by \eqref{eq:Lflow.1} with associated  metric $g_t$ as in \eqref{eq:Lflow.3}. 
    Let $Rm_t$ denote the Riemann curvature tensor of $g_t$ and let $Rm_0^{\cD_0}$ denote the curvature of the transverse connection on $\cD_0$ induced by the Levi-Civita connection of $g_0$.  Then
\[
|Rm_t|_{g_t}^2= |Rm_0^{\cD_0}|_{g_0}^2+c_0\epsilon^4 (1-8\epsilon^2t)^{-2}
\]
    for some constant $c_0>0$.
\end{prop}

\begin{remark}
We observe, by Proposition \ref{prop:curvature.t.2}, that $|Rm_{t}|_{g_t}$ blows up as $t\to\frac{1}{8\epsilon^2}$ (the finite-time singularity), but that $(1-8\epsilon^2t)|Rm_{t}|_{g_t}$ remains bounded as $t\to\frac{1}{8\epsilon^2}$.  We also notice that $|Rm_t|_{g_t}$ remains bounded as $t\to-\infty$, but that $-t|Rm_t|_{g_t}$ always blows up at $t\to-\infty$, unless the tranverse Calabi--Yau geometry is flat. 
\end{remark}

\subsection{Full torsion tensor}

We now study the behaviour of the full torsion tensor along our solution of the Laplacian flow.

\begin{prop}\label{prop:torsion.t.2}
The full torsion tensor $T_t$ of the solution to the Laplacian flow in \eqref{eq:Lflow.1}--\eqref{eq:Lflow.4} is
\begin{equation}\label{eq:torsion.t.2}
    \begin{split}
T_t&=\frac{1}{4}(\tau_0)_tg_t-\frac{1}{4}\rj_{\varphi_t}((\tau_3)_t)\\        
&=\frac{3}{14}\epsilon(1-8\epsilon^2t)^{-1/2} g_t\\
&\qquad-\frac{1}{14}\rj_{\varphi_t}(4\epsilon^2(1-8\epsilon^2t)^{-1}\eta_0\wedge\omega_0-3\epsilon(1-8\epsilon^2t)^{-1/2}\Re\Upsilon_0)\\
&=-\frac{3}{2}\epsilon^3(1-8\epsilon^2t)^{-3/2}\eta_0^2+\frac{1}{2}\epsilon(1-8\epsilon^2t)^{-1/2}g_{\cD_0}.
    \end{split}
\end{equation}
\end{prop}

\begin{proof}
The first two descriptions of $T_t$ in \eqref{eq:torsion.t.2} follow from \eqref{eq:torsion.coclosed}, Lemma \ref{lem:torsion.t} and \eqref{eq:Lflow.ODE.sol}.

We then see from \eqref{eq:j.eta.omega.1}--\eqref{eq:j.eta.omega.3} that
\begin{equation}\label{eq:j.eta.omega}
\rj_{\varphi_t}(\eta_0\wedge\omega_0)=6\epsilon(1-8\epsilon^2t)^{-1/2}\eta_0^2+\frac{2}{\epsilon}(1-8\epsilon^2t)^{1/2}g_{\cD_0},    
\end{equation}
and applying \eqref{eq:j.ReUpsilon.1}--\eqref{eq:j.ReUpsilon.3} yields
\begin{equation}\label{eq:j.ReUpsilon}
\rj_{\varphi_t}(\Re\Upsilon_0)=4g_{\cD_0}.    
\end{equation}
Combining \eqref{eq:j.eta.omega} and \eqref{eq:j.ReUpsilon} shows that
\begin{multline}\label{eq:j.torsion.t.2}
\rj_{\varphi_t}(4\epsilon^2(1-8\epsilon^2t)^{-1}\eta_0\wedge\omega_0-3\epsilon(1-8\epsilon^2t)^{-1/2}\Re\Upsilon_0)\\
=24\epsilon^3(1-8\epsilon^2t)^{-3/2}\eta_0^2-4\epsilon(1-8\epsilon^2t)^{-1/2}g_{\cD_0}.
\end{multline}
Substituting \eqref{eq:Lflow.3} and \eqref{eq:j.torsion.t.2} into \eqref{eq:torsion.t.2} gives the result.
\end{proof}

Our next result describes the norm of the full torsion tensor and is gradient along the flow.

\begin{prop}
Let $T_t$ be the full torsion tensor of the solution to the Laplacian flow in \eqref{eq:Lflow.1}--\eqref{eq:Lflow.4}.  Then
\begin{align*}
    |T_t|^2_{g_t}&=\frac{15}{4}\epsilon^2(1-8\epsilon^2t)^{-1},\\
    |\nabla_tT_t|_{g_t}^2&=c_0\epsilon^4(1-8\epsilon^2t)^{-2},\\
    \div_tT_t&=0,
\end{align*}
where $c_0>0$ is a constant, $\nabla_t$ is the Levi-Civita connection of $g_t$ and $\div_t$ is the divergence with respect to the metric $g_t$.
\end{prop}

\begin{proof}  The proof is almost identical to that of Proposition \ref{prop:norm.full.torsion.t}, so we only prove the formula for $|T_t|_{g_t}^2$.
We observe from the formula \eqref{eq:Lflow.3} for the metric $g_t$ that the transverse metric is not changing along the flow, but that 
\[
|\epsilon^2\eta_0^2|_{g_t}^2=(1-8\epsilon^2 t)^2.
\]
Hence, we see from Proposition \ref{prop:torsion.t.2} that
\[
|T_t|_{g_t}^2=\frac{9}{4}\epsilon^2(1-8\epsilon^2t)^{-1}+\frac{3}{2}\epsilon^2(1-8\epsilon^2t)^{-1}=\frac{15}{4}\epsilon^2(1-8\epsilon^2t)^{-1}
\]
as claimed.  The equation for $|\nabla_tT_t|_{g_t}^2$ and the vanishing of $\div_tT_t$ follow exactly as in the proof of Proposition \ref{prop:norm.full.torsion.t}.
\end{proof}

\begin{remark}
We again see that $|T_t|_{g_t}^2$ and $|\nabla_tT_t|_{g_t}$ blow up at the same rate as we approach the finite time singularity at $t=\frac{1}{8\epsilon^2}$, as we would expect, and again that there is a component of $|Rm_t|_{g_t}$ that instead stays bounded as $t\to\frac{1}{8\epsilon^2}$. Moreover, multiplying each of those quantities by $(1-8\epsilon^2t)$ leads to functions that are bounded as $t\to\frac{1}{8\epsilon^2}$, which will be significant later.
\end{remark}

\subsection{Singularity analysis}

As for the Laplacian coflow, we can now put our estimates together to obtain a description of the finite time singularity in the language of Definition \ref{dfn:sing.types}, which proves Theorem \ref{thm:sings}(b).

\begin{prop}
In a cCY setup \eqref{eq: standard setup}, suppose moreover $M^7$ is compact, and let
\textcolor{black}{$K:=\sup_M|Rm_0^{\cD_0}|_{g_0}$}. Then there is a constant $c_0>0$, independent of $\epsilon$, such that the quantity $\Lambda(t)$ in \eqref{eq:Lambda.t}, along the Laplacian flow solution given by Theorem \ref{thm:Laplacian.flow.cCY}, is given by
\[
\Lambda(t)=(K^2+c_0\epsilon^4(1-8\epsilon^2t)^{-2})^{1/2}.
\]
Hence, the Laplacian flow has a Type I finite-time singularity at $t=\frac{1}{8\epsilon^2}$.
\end{prop}

\begin{proof}
The formula for $\Lambda(t)$ follows from Propositions \ref{prop:curvature.t.2} and \ref{prop:torsion.t.2}.  We then see that $(\frac{1}{8\epsilon^2}-t)\Lambda(t)$ remains bounded as $t\to\frac{1}{8\epsilon^2}$, and thus the result follows from Definition \ref{dfn:sing.types}.
\end{proof}


\section{The Hitchin flow solution}
\label{sec: Hitchin.flow}

In this section we solve the Hitchin flow \eqref{eq:Hitchin.flow} for the natural coclosed $\gt$-structures on a contact Calabi--Yau 7-manifold and show a relationship to Calabi--Yau structures in (real) dimension 8.

\subsection{Solving the flow}

As the initial condition for the Hitchin flow, we again choose the natural coclosed $\gt$-structure $\varphi_0$ as in \eqref{eq: phi0 and psi0} on a contact Calabi-Yau $7$-manifold.  

\begin{prop}
\label{prop:Hitchin.flow}
    On a cCY setup \eqref{eq: standard setup}, the following family of coclosed $\gt$-structures $\varphi_t$ with associated 4-form $\psi_t$, metric $g_t$, and volume $\vol_t$, satisfies the Hitchin flow \eqref{eq:Hitchin.flow}, with initial condition  \eqref{eq: phi0 and psi0}:
\begin{align}
    \varphi_t&=\epsilon(1+\frac{5}{2}\epsilon t)^{-1/5}\eta_0\wedge\omega_0+(1+\frac{5}{2}\epsilon t)^{3/5}\Re\Upsilon_0;\label{eq:H.flow.1}\\
    \psi_t&=\frac{1}{2}(1+\frac{5}{2}\epsilon t)^{4/5}\omega_0^2-\epsilon\eta_0\wedge\Im\Upsilon_0;\label{eq:H.flow.2}\\
    g_t&=\epsilon^2(1+\frac{5}{2}\epsilon t)^{-6/5}\eta_0^2+(1+\frac{5}{2}\epsilon t)^{2/5}g_{\cD_0};\label{eq:H.flow.3}\\
    \vol_t&=\epsilon(1+\frac{5}{2}\epsilon t)^{3/5}\eta_0\wedge\vol_{\cD_0},
\end{align}
    defined  for all $t\in (-\frac{2}{5\epsilon},\infty)$.
\end{prop}

\begin{proof}
We consider the ansatz for our $\gt$-structures $\varphi_t$ solving the Hitchin flow as in \eqref{Eq.fluxo1}, depending on functions $f_t,h_t$ of time only, with dual $4$-forms $\psi_t$ as in \eqref{eq.fluxo2}, with the constraints on $f_0,h_0$ as in \eqref{eq:init.conds.fh}.

Under this assumption, we can compute
\begin{equation*}
    \rd\varphi_t=f_th_t^2\omega_0^2\qandq \rd\psi_t=0,
\end{equation*}
so that the Hitchin flow \eqref{eq:Hitchin.flow} is then equivalent to solving
\begin{equation}\label{eq:Hitchin.ODEs}
    \frac{d}{dt}(h_t^4)=2f_th_t^2\qandq \frac{d}{dt}(f_th_t^3)=0
\end{equation}
using \eqref{eq.fluxo2}.  The second equation in \eqref{eq:Hitchin.ODEs}, together with \eqref{eq:init.conds.fh}, yields
\begin{equation}\label{eq:Hitchin.ODEs1}
    f_t=\epsilon h_t^{-3}.
\end{equation}
Substituting \eqref{eq:Hitchin.ODEs1} into the first equation in \eqref{eq:Hitchin.ODEs} then gives
\begin{equation}
\label{eq:Hitchin.ODEs2}
   \frac{d}{dt}(h_t^4)=2\epsilon h_t^{-1}.
\end{equation}
The ODE \eqref{eq:Hitchin.ODEs2} can be solved using \eqref{eq:init.conds.fh} and thus, using \eqref{eq:Hitchin.ODEs1}, we can conclude that the solution to \eqref{eq:Hitchin.ODEs} is
\begin{align*}
    f_t&=\epsilon(1+\frac{5}{2}\epsilon t)^{-3/5}\qandq h_t=(1+\frac{5}{2}\epsilon t)^{1/5}.
    \qedhere
\end{align*}
\end{proof}

\subsection{Calabi--Yau structure}

We can now show that the solution to the Hitchin flow in Proposition \ref{prop:Hitchin.flow} leads to a Calabi--Yau structure on the spacetime track of the flow.

\begin{lemma}
\label{lem:CY} 
    The following 2-form $\widehat{\omega}$, metric $\widehat{g}$ and complex $4$-form $\widehat{\Upsilon}$ define a Calabi--Yau structure on
    \textcolor{black}{$(-\frac{2}{5\epsilon},\infty)\times M^7$}:
\begin{align}
    \widehat{\omega}
    &=\epsilon(1+\frac{5}{2}\epsilon t)^{-3/5}\rd t\wedge\eta_0+(1+\frac{5}{2}\epsilon t)^{2/5}\omega_0;
    \label{eq:CY4.1}\\
    \widehat{g}
    &=\rd t^2+\epsilon^2(1+\frac{5}{2}\epsilon t)^{-6/5}\eta_0^2+(1+\frac{5}{2}\epsilon t)^{2/5}g_0;
    \label{eq:CY4.2}\\
    \widehat{\Upsilon}
    &=((1+\frac{5}{2}\epsilon t)^{3/5}\rd t+i\epsilon\eta_0)\wedge \Upsilon_0.\label{eq:CY4.3}
\end{align}
The metric $\widehat{g}$ is incomplete and has holonomy contained in $\mathrm{SU}(4)$.
\end{lemma}

\begin{proof}
It is straightforward to check that $\widehat{\omega}$ in \eqref{eq:CY4.1} and $\widehat{\Upsilon}$ in \eqref{eq:CY4.3} induce the metric $\widehat{g}$ in \eqref{eq:CY4.2}, they satisfy the algebraic conditions to define an $\mathrm{SU}(4)$ structure and they are both closed, and thus define a Calabi--Yau structure as claimed.  The incompleteness of the metric is clear.
\end{proof}

\begin{remark}
We know that the metric  induced by the Hitchin flow is given by $\widehat{g}=\rd t^2+g_t$, where $g_t$ is given in \eqref{eq:H.flow.3}, which agrees with the formula in \eqref{eq:CY4.2}. 

We also know by the definition of the Hitchin flow that the induced torsion-free $\mathrm{Spin}(7)$-structure on \textcolor{black}{ $(-\frac{2}{5\epsilon},\infty)$} is given by the 4-form
\[
\Phi=\rd t\wedge\varphi_t+\psi_t,
\]
where $\varphi_t$, $\psi_t$ are given in \eqref{eq:H.flow.1}--\eqref{eq:H.flow.2}. It is straightforward to check from \eqref{eq:CY4.1} and \eqref{eq:CY4.3} that
\[
\Phi=\frac{1}{2}\widehat{\omega}^2+\Re\widehat{\Upsilon}.
\]
Thus the Hitchin flow solution in Proposition \ref{prop:Hitchin.flow} does indeed induce  the Calabi--Yau structure in Lemma \ref{lem:CY}, which then completes the proof of Theorem \ref{thm:Hitchin.flow}.
\end{remark}

\addcontentsline{toc}{section}{References}
\bibliography{Bibliografia-2020-07}
	
\end{document}